\documentclass{siamltex} %[final]

\usepackage{graphicx} \usepackage{amsmath,amssymb} \usepackage{booktabs}
\usepackage{color} \usepackage{subfigure}

\renewcommand{\d}{\mathrm{d}} 
\newcommand{\enorm}[1]{\|#1\|_2} % {}[][]{} fuer optionales Argument
\newcommand{\eqqcolon}{=\mathrel{\mathop:}}
\newcommand{\coloneqq}{\mathrel{\mathop:}=} \newcommand{\spec}[1]{\mathrm{#1}}
\newcommand{\e}{\mathrm{e}} \newcommand{\f}{\mathrm{f}}

\title{A variational principle\\for computing slow invariant manifolds\\in
  dissipative dynamical systems\thanks{This work was supported by the German
    Research Foundation (DFG) through the Collaborative Research Center (SFB)
    568 and the Landesstiftung Baden-W\"urttemberg.}}

\author{Dirk Lebiedz\footnotemark[2] \footnotemark[3] \and Jochen
  Siehr\footnotemark[3] \and Jonas Unger\footnotemark[2]}

\begin{document}
% \figmarkon
\maketitle

% Footnotes
\renewcommand{\thefootnote}{\fnsymbol{footnote}} \footnotetext[2]{Center for
  Systems Biology (ZBSA), University of Freiburg,
  Habsburgerstra\ss{}e 49, 79104 Freiburg, Germany
  (\texttt{dirk.lebiedz@biologie.uni-freiburg.de}).}
\footnotetext[3]{Interdisciplinary Center for Scientific Computing (IWR),
  University of Heidelberg, Im Neuenheimer Feld 368, 69120 Heidelberg,
  Germany.}  \renewcommand{\thefootnote}{\arabic{footnote}}

\begin{abstract}
  A key issue in dimension reduction of dissipative dynamical
  systems with spectral gaps is the identification of slow invariant
  manifolds. We present theoretical and numerical results for a
  variational approach to the problem of computing such manifolds for kinetic
  models using trajectory optimization. The corresponding objective functional
  reflects a variational principle that characterizes trajectories
  on, respectively near, slow invariant manifolds. For a
  two-dimensional linear system and a common nonlinear test
  problem we show analytically that the variational approach asymptotically
  identifies the exact slow invariant manifold in the limit of both an
  infinite time horizon of the variational problem with fixed spectral gap
  and infinite spectral gap with a fixed finite time
  horizon. Numerical results for the linear and nonlinear model problems as
  well as a more realistic higher-dimensional chemical reaction mechanism are
  presented.
\end{abstract}

\begin{keywords}
  Model reduction, slow invariant manifold, optimization, calculus of
  variations, extremum principle, curvature, chemical kinetics
\end{keywords}

\begin{AMS}
  37N40, 37M99, 80A30, 92E20
\end{AMS}

\pagestyle{myheadings} \thispagestyle{plain} \markboth{D.~LEBIEDZ, J.~SIEHR,
  AND J.~UNGER}{A VARIATIONAL PRINCIPLE FOR COMPUTING SLOW MANIFOLDS}

\section{Introduction} \label{s:intro}
In dissipative ordinary differential equation systems modeling chemical
reaction kinetics the phase flow generally causes anisotropic volume
contraction due to multiple time scales with spectral gaps. This leads to a
bundling of trajectories near ``invariant manifolds of slow motion'' of
successively lower dimension during time evolution. Model reduction methods
exploit this for simplifying the underlying ordinary differential equation
models via time scale separation into fast and slow modes and eliminating
the fast modes by enslaving them to the slow ones as a graph of a function
which defines the slow invariant (attracting) manifold (SIM).

Early model reduction approaches in chemical kinetics like the quasi
steady-state and partial equilibrium assumption \cite{Warnatz2006} have been
performed ``by hand'', modern numerical approaches are supposed to
automatically compute
a reduced model without need for detailed expert knowledge of chemical
kinetics by the user. Many of these techniques are based on an explicit
time-scale analysis of the underlying ordinary differential equation (ODE)
system.

Among those methods that became popular in applications are the intrinsic low
dimensional manifold (ILDM) method \cite{Maas1992} and recent extensions of
its main ideas, e.g.\ the global quasi-linearization (GQL) \cite{Bykov2008},
computational singular perturbation (CSP) \cite{Lam1985, Lam1994}, Fraser's
algorithm \cite{Davis1999, Fraser1988, Nguyen1989}, the method of
invariant grids \cite{Chiavazzo2009, Gorban2005, Gorban2005a}, the constrained
runs algorithm \cite{Gear2005, Zagaris2009}, rate-controlled constrained
equilibrium (RCCE) \cite{Keck1971}, the invariant constrained equilibrium edge
preimage curve (ICE-PIC) method \cite{Ren2005,Ren2006a}, flamelet-generated
manifolds \cite{Delhaye2007, Oijen2000}, and finite time Lyapunov exponents
\cite{Mease2003}. For a comprehensive overview see e.g.\ \cite{Gorban2005} and
references therein.

Reaction trajectories in phase space that are solutions of an ODE system
$\dot{x}(t)=f(x(t)), x(0)=x_0, f \in C^{\infty}$, describing chemical
kinetics are uniquely determined by their initial values and the corresponding
orbits bear global information about phase space structure. Based on
Lebiedz' idea to search for an extremum principle that distinguishes
trajectories on or near slow attracting manifolds, an optimization approach
for computing such trajectories has been applied in \cite{Lebiedz2004c,
  Lebiedz2006b, Reinhardt2008}. In \cite{Lebiedz2009} the authors propose and
discuss various geometrically motivated optimization criteria for the
formulation of a suitable extremum principle and present numerical results for
several applications.

The present work systematically analyzes a variational formulation
of the problem to compute slow invariant manifolds and its potential
for identifying the correct manifold for linear and nonlinear test problems in
two-dimensional phase space. We analytically prove the correct identification of
the slow eigenspace and SIM respectively in the limit of infinite-time horizon
of the variational problem and derive
an error quantification as a function of spectral gap and finite time
horizon length. In addition, we provide corresponding numerical results
confirming the theoretical prediction.

\section{Variational Problem} \label{s:optim_appr}
We consider autonomous ODE systems of the form $\dot{x}=f(x)$ modeling
chemical reaction kinetics that have a stable fixed point corresponding to
chemical equilibrium. The basic idea of our approach is the formulation
of a variational principle that captures essential properties of a slow
invariant manifold (SIM). We propose an appropriate
characterization of maximum ``slowness'' in terms of an integral over suitably
defined curvature (velocity change) of trajectories measured in the Euclidean
norm. The SIM is generally characterized by the property that all trajectories
in its neighborhood converge
faster to the manifold than to the attractor, the chemical equilibrium
point. Adrover et al.\ \cite{Adrover2007} recently argued that this might be
interpreted as a ratio $r>1$ of the local stretching (contraction)
rate of vectors orthogonal to the manifold compared to those tangent to the
manifold. This point of view comes close to our reasoning on the basis of a
variational principle.

\subsection{Trajectory-Based Optimization Approach} \label{ss:gen_appr}
The variational problem can be formulated as
\begin{subequations} \label{eq:op}
  \begin{equation} \label{eq:op:of} \min_{x(t)} \int_{t_{0}}^{t_{\f}}
    \Phi\left(x(t)\right) \; \d t
  \end{equation}
  \textrm{subject to}
  \begin{align}
    \frac{\d x(t)}{\d t} &= f\left(x(t)\right) \label{eq:op:dyn}\\
    0 &= g\left(x(t_{*})\right)  \label{eq:op:cr}\\
    x_j(t_{*}) &= x_j^{t_*},\quad j \in I_{\text{fixed}}, \label{eq:op:pv}
  \end{align}
\end{subequations}
with $t_0 \leqslant t_{*} \leqslant t_{\f}$. The variable $x=(x_i)_{i=1}^n$ denotes the
state vector and $I_{\mathrm{fixed}}$ is an index set that contains the
indices of state variables (denoted as \emph{reaction progress variables} in
chemical kinetics) with fixed values at fixed time $t_*$ chosen to
parameterize the
reduced model, i.e.\ the slow attracting manifold to be computed. Thus, those
state variables representing the actual degrees of freedom within the
optimization problem are $x_j(t_*), j\notin I_{\rm fixed}$. The process of
determining $x^{t_*}_j, j \notin I_{\rm fixed}$ from $x^{t_*}_j,j \in I_{\rm
  fixed}$ is known as \emph{species reconstruction} in chemical kinetics and
represents a function mapping the reaction
progress variables to the full species composition by determining a point on
the slow attracting manifold. The system dynamics (e.g.\ chemical kinetics
determined by the reaction mechanism) are described by (\ref{eq:op:dyn}) and
enter the optimization problem as equality constraints. Hence an optimal
solution of (\ref{eq:op}) always satisfies the system dynamics of the full ODE
system and therefore represents a solution trajectory of
(\ref{eq:op:dyn}). Additional constraints (e.g.\ chemical element mass
conservation relations in the case of chemical kinetics that have to
be obeyed due to the law of mass conservation) are collected in the function
$g$  in (\ref{eq:op:cr}). The state variables chosen as parameterization of
the reduced model (slow invariant manifold) are fixed via the equality
constraint (\ref{eq:op:pv}) at $t_*$. The
objective functional $\Phi(x(t))$ in (\ref{eq:op:of}) characterizes the
variational principle that will be discussed in the next section.

\subsection{Optimization Criterion} \label{ss:optim_crit}
In \cite{Lebiedz2009, Reinhardt2008}
\begin{equation*}
  \Phi(x) \coloneqq \enorm{J_f(x)\;f}
\end{equation*}
is proposed as a suitable criterion with $J_f(x)$ being the Jacobian of
the right hand side $f$ evaluated at $x(t)$ and $\|\cdot\|_2$ denoting the
Euclidean norm.

The term $J_f(x)\;f$ represents the rate of change of reaction velocity
in its own direction along a trajectory and can be interpreted as a specific
definition of curvature in time parameterization of the curve
\begin{equation*}
  \ddot{x} = \frac{\d \dot{x}}{\d t}
  = \frac{\d\dot{x}}{\d x} \cdot \frac{\d x}{\d t} = J_f(x) \cdot f .
\end{equation*}
The minimization of the time integral over $\Phi$ in (\ref{eq:op:of})
incorporates the ``maximum slowness'' issue in terms of an average over
suitably measured local curvature of a trajectory.

For further analytical and numerical investigation of the variational
formulation we consider a slight modification of the objective
functional
\begin{equation} \label{eq:objA2} \int_{t_0}^{t_{\f}}\enorm{J_f(x)\;f}^2 \; \d
  t = \int_{t_0}^{t_{\f}}f^{\rm T} J_f(x)^{\rm T} J_f(x) f \; \d t .
\end{equation}

\subsection{Forward and Reverse Mode} \label{ss:forw_rev} In previous
publications \cite{Lebiedz2004c, Lebiedz2006b, Lebiedz2009} the general
optimization problem (\ref{eq:op}) is formulated with $t_*=t_0=0$ and for the
numerical computations $t_\f$ is chosen ``large enough'' for $x(t_\f)$ to
be close to the attractor, the chemical equilibrium point.
The numerical value $t_0=0$ is arbitrary as the ODE is autonomous.

In contrast to this ``forward formulation'', in the present work additionally
the ``backward formulation'' $t_\f=t_*=0$ is used. In fact, this is the more
natural formulation for the identification of a trajectory on the slow
invariant manifold which stays on this manifold during backward time
evolution. However, the solution of the backward problem is much more
challenging numerically since it is highly unstable and ill-conditioned for a
dissipative dynamical system.
We will refer to the first case with $t_0=t_*=0$ as \emph{forward mode} and to
the latter ($t_\f=t_*=0$) as \emph{reverse mode}. Both modes can
be seen as special cases of the general formulation (\ref{eq:op}). We deal with
the numerical instability of the
\emph{reverse mode} by a collocation approach (see Section \ref{ss:num_meth})
with a fine discretization grid for the objective functional (\ref{eq:op:of})
and the differential equation constraint
(\ref{eq:op:dyn}) and apply robust interior point optimization methods
\cite{Waechter2006} to solve the resulting high-dimensional nonlinear
programming problem (NLP).

\section{Theoretical Results} \label{s:theory} In this section theoretical
results for the solution of the \emph{reverse mode} problem formulation are
presented. For a general two-dimensional linear system with distinct negative
real eigenvalues and the nonlinear Davis--Skodje test model \cite{Davis1999,
  Singh2002}, it is shown that for infinite time horizon of the
variational problem the exact slow manifold is identified by the solution of
the previously introduced variational problem.

\subsection{Linear Model} \label{ss:theory_lin} We consider the
two-dimensional linear model
\begin{equation} \label{eq:lin}
  \begin{aligned}
    \dot{y}_1(t) &= -\lambda\;y_1(t) \\
    \dot{y}_2(t) &= (-\lambda-\gamma)\;y_2(t)
  \end{aligned}
\end{equation}
with two time scales
$\mathcal{O}(\lambda)$ and $\mathcal{O}(\lambda+\gamma)$ where
$\gamma>0$ measures the spectral gap (stiffness) of the system. In
order to allow for a parameterization of the slow eigenspace by both state
variables $y_1$ and $y_2$, we apply an orthogonal transformation via rotation
matrices $R$.
\begin{equation*}
  R = \begin{pmatrix}
    \cos \frac{\pi}{4}  & -\sin \frac{\pi}{4} \\
    \sin \frac{\pi}{4}  &  \cos \frac{\pi}{4}
  \end{pmatrix}.
\end{equation*}
Hence system (\ref{eq:lin}) is transformed to $\dot{x}=A x$ with
\begin{equation} \label{eq:lin_trans}
  A = \begin{pmatrix}
    -\lambda-\frac{\gamma}{2} & \frac{\gamma}{2} \\
     \frac{\gamma}{2} & -\lambda-\frac{\gamma}{2}
  \end{pmatrix}
\end{equation}
and the slow eigenspace is the first bisectrix $x_1 \equiv x_2$. Since
orthogonal transformations purely rotate the phase portrait of the dynamical
system, the following considerations capture the general linear
two-dimensional case.\\

\begin{theorem}\label{thm:lin}
  Let $\dot{x}=Ax$ be a two-dimensional linear model, $A$ as in
  (\ref{eq:lin_trans}) with distinct (real-valued) eigenvalues
  $-\lambda$ and $-(\lambda+\gamma)$, $\gamma \in \mathbb{R}^+$, fast and slow
  eigenspaces $\Lambda_{\mathrm{f}}$ and $\Lambda_{\mathrm{s}}$ corresponding
  to $-(\lambda+\gamma)$ and $-\lambda$, respectively.
  Let $x^{*}$ be the optimal solution of (\ref{eq:op}) with
  $t_{*}=t_{\f}\in\mathbb{R}$, $g \equiv 0$, $f(x)=Ax$, and $\Phi \left( x(t)
  \right) = \left\|J_f(x(t))\;f(x(t))\right\|^2_2 =
  \left\|AAx(t)\right\|^2_2$.

Then for all $\gamma>0$ and $t_0<t_\f$ it holds
  \begin{equation*}
    \lim_{t_0 \rightarrow -\infty} d(x^*(t_{\mathrm f}), \Lambda_{\mathrm s})
    = \lim_{t_0 \rightarrow -\infty} \inf_{b \in \Lambda_{\mathrm s}}
    \|x^*(t_{\mathrm f}) - b \|_2 = 0.
  \end{equation*}
\end{theorem}

\begin{proof}
  We assume w.l.o.g.\ the second variable being the progress variable, i.e.\
$I_\textrm{fix}=\{2\}$, and $\lambda=1$. The objective criterion
$\Phi(x(t))$ can be computed as
  \begin{equation}\label{eq:crit_lin}
    \begin{aligned}
      \left\|AAx(t)\right\|_2^2 =
      & ( x_1(t) )^2 \left( 1+2\gamma+3\gamma^2+2\gamma^3+\frac{\gamma^4}{2} \right)\\
      &+ ( x_2(t) )^2 \left( 1+2\gamma+3\gamma^2+2\gamma^3+\frac{\gamma^4}{2} \right)\\
      &+ x_1(t)\; x_2(t)\left(-4\gamma-6\gamma^2-4\gamma^3-\gamma^4 \right).
    \end{aligned}
  \end{equation}
  The general solution of the ODE $\dot{x}=Ax$ is
  \begin{subequations}
    \label{eq:results:SolutionOfODE}
    \begin{align}
      \label{eq:results:SolutionOfODEa}
      x_1(t) &= c_1\e^{-t}+c_2\e^{(-1-\gamma) t}\\
      \label{eq:results:SolutionOfODEb}
      x_2(t) &= c_1\e^{-t}-c_2\e^{(-1-\gamma) t}.
    \end{align}
  \end{subequations}
  Solution (\ref{eq:results:SolutionOfODE}) is substituted into criterion
  (\ref{eq:crit_lin}) and integration over time yields the objective functional
  \begin{align}
    \int_{t_0}^{t_{\f}}\!\left\|AAx(t)\right\|^2_2\,\textrm{d}t &= \int_{t_0}^{t_{\f}}\!\left[2c_1^2\e^{-2t}+\left(2+8\gamma +12\gamma^2+8\gamma^3+2\gamma^4\right)c_2^2\e^{\left(-1-\gamma\right)2t}\right]\,\textrm{d}t\notag\\\notag\\
    \label{eq:results:c_1^2}
    &= c_1^2\left(\e^{-2t_0}-\e^{-2t_{\f}}\right)-\xi
    c_2^2\left(\e^{\left(-1-\gamma\right)2t_0}-\e^{\left(-1-\gamma\right)2t_{\f}}\right),
  \end{align}
  with $\xi = \tfrac{2+8\gamma +12\gamma^2+8\gamma^3+2\gamma^4}{-2-2\gamma} <
  0$.  An expression $c_1(c_2)$ for $c_1$ as a function of $c_2$ can be
  computed from
  (\ref{eq:results:SolutionOfODEb}) which only dependents on $c_2$ because of
  the fixed final value of $x_2(t_\f)$:
  \begin{equation*}
    x_2(t_{\f}) = c_1\e^{-t_{\f}}-c_2\e^{\left(-1-\gamma\right)t_{\f}} \quad \Longrightarrow \quad c_1(c_2) = \frac{x_2\left(t_{\f}\right)+c_2\e^{\left(-1-\gamma\right)t_{\f}}}{\e^{-t_{\f}}}.
  \end{equation*}
  %That gives
  %\begin{equation*}
  %  c_1^2(c_2) = \frac{\left(x_2^{t_{\f}}\right)^2+c_2^2\e^{\left(-1-\gamma\right)2t_{\f}}+2x_2^{t_{\f}}c_2\e^{\left(-1-\gamma\right)t_{\f}}}{\e^{-2t_{\f}}}.
  %\end{equation*}
  This formula can be used to eliminate $c_1$ from (\ref{eq:results:c_1^2})
  leading to an
  expression $h(c_2)$ only depending on $c_2$ (and $t_0$, $t_\f$, $\gamma$,
  which are assumed to be fixed at the moment)
  \begin{align*}
    h\left(c_2\right)\coloneqq\ &  \frac{\left(x_2^{t_{\f}}\right)^2\e^{-2t_0}}{\e^{-2t_{\f}}}+\frac{\e^{\left(-1-\gamma\right)2t_{\f}}\e^{-2t_0}}{\e^{-2t_{\f}}}c_2^2+\frac{2x_2^{t_{\f}}\e^{\left(-1-\gamma\right){t_{\f}}}\e^{-2t_0}}{\e^{-2t_{\f}}}c_2 \\
    &-\left(x_2^{t_{\f}}\right)^2-\e^{\left(-1-\gamma\right)2t_{\f}}c_2^2-2x_2^{t_{\f}}\e^{\left(-1-\gamma\right)t_{\f}}c_2 \\
    &-\xi \e^{\left(-1-\gamma\right)2t_0}c_2^2
    +\xi\e^{\left(-1-\gamma\right)2t_{\f}}c_2^2,
  \end{align*}
  which should be minimal for identification of the optimal $c_2$. The first
  order necessary condition for a minimum
  $\tfrac{\textrm{d}h(c_2)}{\textrm{d}c_2} = 0$ gives a solution
  \begin{equation*}
    \hat{c}_2=\frac{x_2^{t_{\f}}\e^{\left(-1-\gamma\right)t_{\f}}-x_2^{t_{\f}}\e^{\left(1-\gamma\right)t_{\f}}\e^{-2t_0}}{\e^{-2\gamma t_{\f}}\e^{-2t_0}-\xi \e^{\left(-1-\gamma \right)2t_0}+(\xi-1)\e^{\left(-1-\gamma\right)2t_{\f}}}.
  \end{equation*}
  Checking the second order sufficient conditions
  \begin{equation*}
    \frac{\textrm{d}^2h}{\textrm{d}c_2^2} \equiv 2\e^{-2\gamma t_{\f}}\left(\e^{-2t_0}-\e^{-2t_{\f}}\right)+2\xi\left(\e^{-2t_{\f}}\e^{-2\gamma t_{\f}}-\e^{-2t_0}\e^{-2\gamma t_0}\right) > 0 \quad \forall c_2,\ t_{\f}>t_0
  \end{equation*}
  guarantees $\hat{c}_2$ being a minimum.

  The solution $\hat{c}_2$ and $c_1(\hat{c}_2)$ are substituted in
  (\ref{eq:results:SolutionOfODEa}) evaluated at fixed final time $t_{\f}$
  yielding an
  expression for $x_1(t_{\f})$ additionally depending on $\gamma$ and $t_0$
  \begin{equation}\label{eq:results:Linear_Model_Result}
    \begin{aligned}
      x_1(t_{\f}) &= c_1(\hat{c}_2)\e^{-t_{\f}}+\hat{c}_2\e^{\left(-1-\gamma\right)t_{\f}} = \frac{x_2^{t_{\f}}+\hat{c}_2\e^{\left(-1-\gamma\right)t_{\f}}}{\e^{-t_{\f}}}\e^{-t_{\f}}+\hat{c}_2\e^{\left(-1-\gamma\right)t_{\f}} \\
      &= x_2^{t_{\f}}\left[1+\underbrace{\left(\frac{2\e^{\left(-1-\gamma\right)2t_{\f}}-2\e^{-2\gamma t_{\f}}\e^{-2t_0}}{\e^{-2\gamma t_{\f}}\e^{-2t_0}-\xi\e^{\left(-1-\gamma\right)2t_0}+\left(\xi-1\right)\e^{\left(-1-\gamma\right)2t_{\f}}}\right)}_{\eqqcolon \chi}\right] \\
      &= x_2^{t_{\f}}\left[1+\left(\frac{2\e^{\left(-1-\gamma\right)2t_{\f}}}{\e^{-2\gamma t_{\f}}\e^{-2t_0}-\xi\e^{\left(-1-\gamma\right)2t_0}+\left(\xi-1\right)\e^{\left(-1-\gamma\right)2t_{\f}}} \right.\right.\\
      &\left.\left.\qquad\qquad\qquad-\frac{2\e^{-2\gamma
              t_{\f}}}{\e^{-2\gamma t_{\f}}-\xi\e^{-2\gamma
              t_0}+\left(\xi-1\right)\e^{\left(-1-\gamma\right)2t_{\f}}\e^{2t_0}}\right)\right]
    \end{aligned}
  \end{equation}
  with error term $\chi$ quantifying the deviation from the slow eigenspace
  $x_1 \equiv x_2$. Finally in the limit
  $t_0\rightarrow -\infty$ it can be seen that
  \begin{align*}
    \lim \limits_{t_0 \to -\infty}x_1(t_{\f}) = x_2^{t_{\f}}
  \end{align*}
  meaning the slow eigenspace $x_1(t) = x_2(t)$ is identified by a solution of
  the optimization problem.\hfill
\end{proof}

In Figure \ref{f:results: Linear_Model_3D_Plot} the error term $\chi$ is
plotted. It illustrates that for increasing spectral gap $\gamma$ and
increasing time intervals $[t_0,t_\f]$ the approximation of the SIM improves
while the error decreases exponentially.
\begin{figure}[htb]
  \centering
  \begin{center}
    \includegraphics[width=10cm]{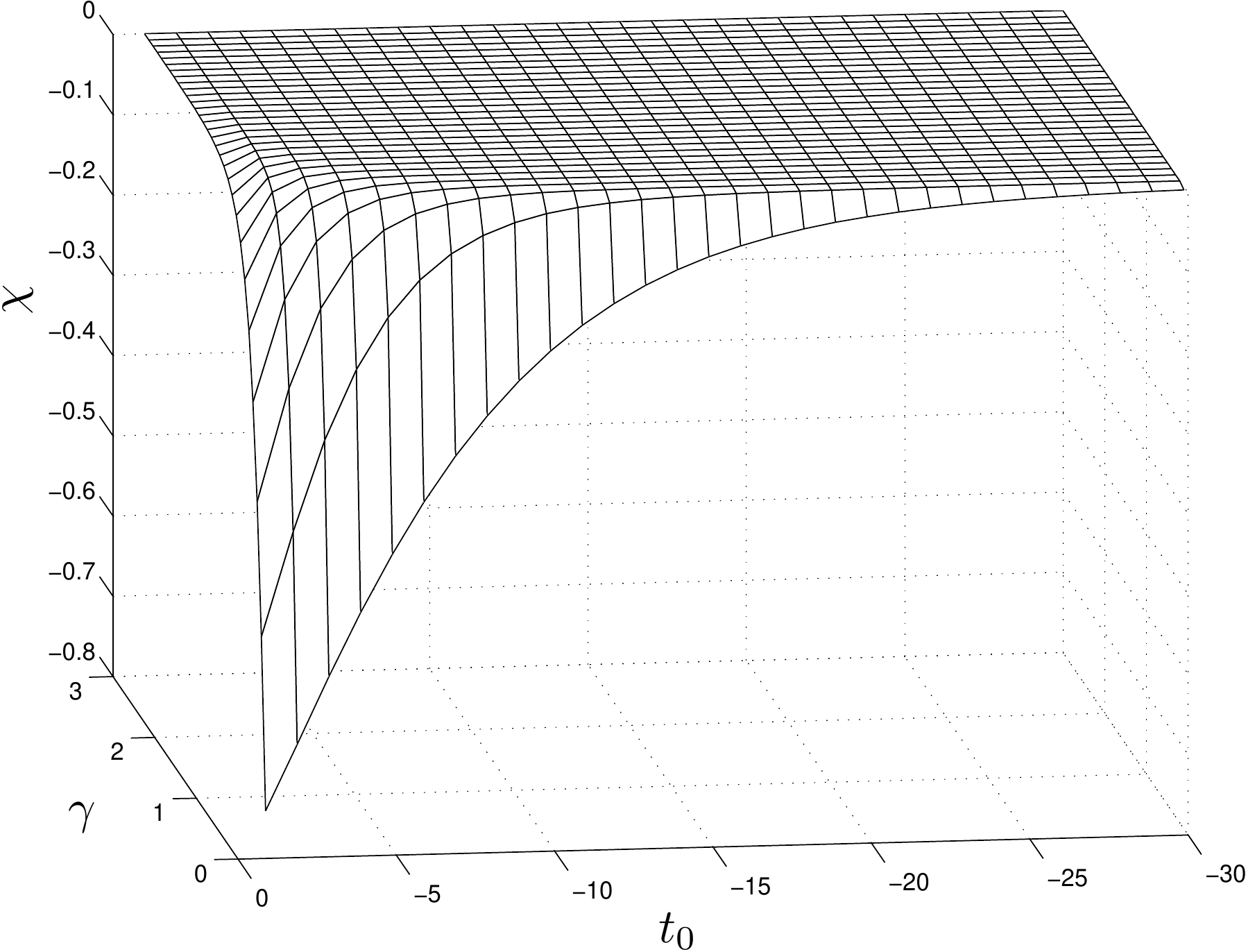}
  \end{center}
  \caption{\label{f:results: Linear_Model_3D_Plot} The error term $\chi$
    in (\ref{eq:results:Linear_Model_Result}) plotted
    against $t_0$ and $\gamma$ with $t_{\f}=0$.}
\end{figure}

\subsection{Davis--Skodje Test Problem} \label{ss:theory_DS}
The Davis--Skodje model (\ref{eq:results:DavisSkodje})
\cite{Davis1999,Singh2002} is widely used for analysis and performance
tests of model reduction techniques supposed to identify slow invariant
manifolds,
\begin{subequations}
  \label{eq:results:DavisSkodje}
  \begin{align}
    \frac{\d x_1}{\d t} &= - x_1 \label{eq:results:DSa}\\
    \frac{\d x_2}{\d t} &= - \gamma x_2 + \frac{(\gamma - 1)x_1 +
      \gamma \label{eq:results:DSb} x_1^2}{(1+x_1)^2},
  \end{align}
\end{subequations}
where $\gamma > 1$ is a measure for the spectral gap (stiffness) of the
system.  Typically model reduction algorithms show a good performance for
large values of $\gamma$, which represent a large time scale separation. Small
values of $\gamma$ impose a significantly harder
challenge on the computation of the slow invariant manifold. For reasons of
adjustable time scale separation and analytically computable SIM, the
Davis--Skodje model is widely used for testing numerical model
reduction approaches. We provide analytical and numerical results for the
variational approach with the Davis--Skodje model. \\

\begin{theorem}\label{thm:DS}
  Let $\dot{x}=f(x)$ be the Davis--Skodje model (\ref{eq:results:DavisSkodje}),
  the slow invariant manifold defined by
  $\Lambda_{\mathrm{s}}\coloneqq\{(x_1,x_2) \in \mathbb{R}^2 \ | \
  x_2=\tfrac{x_1}{1+x_1}$\} and $x^{*}$ the optimal solution of (\ref{eq:op})
  with $t_{*}=t_{\f}\in\mathbb{R}$, $g \equiv 0$, $I_\mathrm{fix} = \{1\}$,
  and $\Phi \left( x(t) \right) = \left\|J_f
    \left(x(t)\right)\;f\left(x(t)\right) \right\|^2_2$. Then for all
  $\gamma>1$, $x_1^{t_\f}>1$ and $t_0<t_\f$ holds
  \begin{equation*}
    \lim_{t_0 \rightarrow -\infty} d(x^*(t_{\mathrm f}), \Lambda_{\mathrm s})
    = \lim_{t_0 \rightarrow -\infty} \inf_{b \in \Lambda_{\mathrm s}}
    \|x^*(t_{\mathrm f}) - b \|_2 = 0.
  \end{equation*}
\end{theorem}

\begin{proof}
  The Jacobian of $f$ is given by
  \begin{equation*}
    J_f(x(t)) = \begin{pmatrix} -1 &  0 \\ \frac{(1+\gamma)x_1(t)+\gamma-1}
      {\left(1+x_1(t)\right)^3} &  -\gamma\end{pmatrix}.
  \end{equation*}
  $\Phi\left(x(t)\right)$ in the objective function can be computed
  explicitly as
  \begin{align*}
    \Phi\left(x(t)\right)=
    &\left(x_1(t)\right)^2+\gamma^4\left(x_2(t)\right)^2+\frac{\left(x_1(t)\right)^2\left(1-2\gamma^2+\gamma^4\right)}{\left(1+x_1(t)\right)^6}\\
    &+\frac{\left(x_1(t)\right)^3\left(-2-2\gamma^2+4\gamma^4\right)}{\left(1+x_1(t)\right)^6}+\frac{\left(x_1(t)\right)^4\left(1+2\gamma^2+6\gamma^4\right)}{\left(1+x_1(t)\right)^6}\\
    &+\frac{\left(x_1(t)\right)^5\left(2\gamma^2+4\gamma^4\right)}{\left(1+x_1(t)\right)^6}+\frac{\left(x_1(t)\right)^6\gamma^4}{\left(1+x_1(t)\right)^6}\\
    &-x_1(t)x_2(t)\frac{\left(2\gamma^2x_1(t)-2\gamma^2+2\gamma^4+4\gamma^4x_1(t)+2\gamma^4\left(x_1(t)\right)^2\right)}{\left(1+x_1(t)\right)^3}.
  \end{align*}

  An analytical solution of model (\ref{eq:results:DavisSkodje}) will be
  computed in the following. The first differential equation yields $x_1(t) =
  c_1\e^{-t}$ as a general solution. Equation (\ref{eq:results:DSb}) is a
  inhomogeneous first order
  linear ordinary differential equation and the ansatz of the method
  of variation of parameters gives $ x_2(t) = x_{2,\textrm{hom}}(t) +
  x_{2,\textrm{part}}(t)$.  The homogeneous equations are solved by
  $x_{2,\textrm{hom}}(t) = c_2\e^{-\gamma t}$, because
  $\dot{x}_{2,\textrm{hom}}(t) = -\gamma c_2\e^{-\gamma t} = -\gamma
  x_{2,\textrm{hom}}(t)$. To determine $x_{2,\textrm{part}}(t)$ the relation
  \begin{equation*}
    \e^{-\gamma t}\dot{c}_2(t) = \frac{\left(\gamma-1\right)c_1\e^{-t}+\gamma c_1^2\e^{-2t}}{\left(1+c_1\e^{-t}\right)^2}\\
  \end{equation*}
  has to be solved for $c_2$:
  \begin{equation*}
    c_2(t) = \int\!\dot{c}_2(t)\,\textrm{d}t = \int\!\frac{\left(\gamma-1\right)c_1\e^{-t}+\gamma c_1^2\e^{-2t}}{\e^{-\gamma t}\left(1+c_1\e^{-t}\right)^2}\,\textrm{d}t
    = \frac{c_1\e^{\gamma t}}{c_1+\e^{t}}.
  \end{equation*}
  Therefore, the missing part is
  \begin{equation*}
    x_{2,\textrm{part}}(t) =
    \e^{-\gamma t}\frac{c_1\e^{\gamma t}}{c_1+\e^{t}} =
    \frac{c_1}{c_1+\e^{t}}
  \end{equation*}
  and the full solution of the ODE is given by
  \begin{subequations}
    \label{eq:results:Solution_ODE_DS}
    \begin{align}
      x_1(t) &= c_1\e^{-t} \label{eq:results:Solution_ODE_DSa}\\
      x_2(t) &= c_2\e^{-\gamma
        t}+\frac{c_1}{c_1+\e^{t}}.\label{eq:results:Solution_ODE_DSb}
    \end{align}
  \end{subequations}
  Now the criterion $\Phi$ is integrated over time using
  (\ref{eq:results:Solution_ODE_DS}). In the next formula $r_i$, $i=1,2$ represents a
  ``rest'' -- all terms independent of $x_2(t)$, hence also independent of
  $c_2$, which annihilate after differentiation with respect to $c_2$
  afterwards. Making use of
  $c_1=x_1(t_{\f})\;\e^{t_{\f}}=x_1^{t_{\f}}\e^{t_{\f}}$ due to
  (\ref{eq:results:Solution_ODE_DSa}) yields as an expression for the
  objective function only depending on $c_2$
  \begin{align*}
    h(c_2)
    =&\ r_1+\int_{t_0}^{t_{\f}}\!\gamma^4\left(x_2(t)\right)^2\,\textrm{d}t \\
    &-\int_{t_0}^{t_{\f}}\!x_1(t)x_2(t)\frac{2\gamma^2x_1(t)-2\gamma^2+2\gamma^4+4\gamma^4x_1(t)+2\gamma^4\left(x_1(t)\right)^2}{\left(1+x_1(t)\right)^3}\,\textrm{d}t\\
    % = &c_2^2\left(\frac{1}{2}\gamma^3\e^{-2\gamma t_0}-\frac{1}{2}\gamma^3\e^{-2\gamma t_{\f}}\right) \\
    % &+\int_{t_0}^{t_{\f}}\!\left[\frac{2\gamma^4c_1c_2\e^{-\gamma t}}{c_1+\e^t}\right.   \\
    % &\left. -c_1c_2\e^{\left(-1-\gamma\right)t}\frac{\left(2\gamma^2c_1\e^{-t}+2\gamma^2+2\gamma^4+4\gamma^4c_1\e^{-t}+2\gamma^4c_1^2\e^{-2t}\right)}{\left(1+c_1\e^{-t}\right)^3}\right]\,\textrm{d}t +r \\
    =&\ r_2 + c_2^2\left(\frac{1}{2}\gamma^3\e^{-2\gamma t_0}-\frac{1}{2}\gamma^3\e^{-2\gamma t_{\f}}\right)\\
    &+c_2\int_{t_0}^{t_{\f}}\!\!
    \underbrace{\left[\frac{2\gamma^4c_1\e^{-\gamma t}}{c_1+\e^t}
        -2\gamma^2c_1\e^{(-1-\gamma)t}\frac{c_1\e^{-t}-1+\gamma^2\left(1+2c_1\e^{-t}+c_1^2\e^{-2t}\right)}{(1+c_1\e^{-t})^3}\right]}_{\eqqcolon
      \varphi(t)}\textrm{d}t
  \end{align*}

  The necessary first-order condition for a minimum is applied. Setting
  $\tfrac{\d h\left(c_2\right)}{\d c_2}=0$ results in an optimal
  \begin{equation*}
    \check{c}_2 \coloneqq \frac{-\int_{t_0}^{t_{\f}}\!\varphi(t)\,\d t}
    {\gamma^3\e^{-2\gamma t_0}-\gamma^3\e^{-2\gamma t_{\f}}}.
  \end{equation*}
  The second order check
  \begin{equation*}
    \frac{\textrm{d}^2h}{\textrm{d}c_2^2} \equiv
    \gamma^3\e^{-2\gamma t_0}-\gamma^3\e^{-2\gamma t_{\f}} > 0
    \quad \forall c_2,\ t_0<t_{\f}, \ \gamma>1
  \end{equation*}
  assures $\check{c}_2$ being a minimum.

  An expression for $x_2(t_{\f})$ can be derived by substituting $c_1$ and
  $\check{c}_2$ in (\ref{eq:results:Solution_ODE_DSb}):
  \begin{equation*}
    x_2\left(t_{\f}\right) =
    \frac{-\int_{t_0}^{t_{\f}}\!\varphi(t)\,\textrm{d}t}{\gamma^3\e^{-2\gamma
        t_0}-\gamma^3\e^{-2\gamma t_{\f}}}\e^{-\gamma
      t_{\f}}+\frac{x_1^{t_{\f}}}{x_1^{t_{\f}}+1}.
  \end{equation*}
  The proof for the relation
  \begin{equation} \label{eq:lem_rel} \lim_{t_0 \to
      -\infty}\frac{-\int_{t_0}^{t_{\f}}\!\varphi(t)\,\textrm{d}t}{\gamma^3\e^{-2\gamma
        t_0}-\gamma^3\e^{-2\gamma t_{\f}}}\e^{-\gamma t_{\f}} = 0
  \end{equation}
  will be given in the following Lemma \ref{lem:DS}. Because of
  (\ref{eq:lem_rel}) it holds
  \begin{align*}
    \lim \limits_{t_0 \to -\infty}x_2(t_{\f}) =
    \frac{x_1^{t_{\f}}}{x_1^{t_{\f}}+1},
  \end{align*}
  which is the analytic expression for the slow invariant manifold of the
  Davis--Skodje system (see \cite{Davis1999}). This completes the proof.
  \hfill\end{proof}\\

% Now the prove of Theorem \ref{thm:DS} is to be replenished by the following
% lemma.
\begin{lemma} \label{lem:DS} Under the conditions of Theorem \ref{thm:DS}
  equation (\ref{eq:lem_rel}) holds.
\end{lemma}

\begin{proof}
In the following $\left(x\right)_n$ is the Pochhammer symbol which is defined as $(x)_n\coloneqq x\left(x+1\right)\cdots\left(x+n-1\right)$. For the
proof the integral in (\ref{eq:lem_rel}) is evaluated using Wolfram Mathematica$^\circledR$ 7:
\begin{equation*}
  \int_{t_0}^{t_{\f}}\!\varphi(t)\,\textrm{d}t
  %&=\left[\frac{-\left(2\gamma^2\e^{\left(1-\gamma\right)t}\left(c_1\left(\gamma-1\right)\left(c_1\gamma+c_1+\gamma\e^t\right)-\gamma^2\left(c_1+\e^t\right)^2\sum\limits_{n=0}^\infty\frac{\left(1\right)_n\left(1-\gamma\right)_n}{\left(2-\gamma\right)_n n!}\left(-\frac{\e^{nt}}{c_1^n}\right)\right)\right)}{c_1\left(\gamma-1\right)\left(c_1+\e^t\right)^2}\right]_{t_0}^{t_{\f}} \\
  %= &\left[\frac{-2\gamma^2\e^{\left(1-\gamma\right)t}c_1\left(\gamma-1\right)}{c_1\left(\gamma-1\right)\left(c_1+\e^t\right)}\right]_{t_0}^{t_{\f}}-\left[\frac{2\gamma^3\e^{\left(1-\gamma\right)t}\left(c_1+\e^t\right)\sum\limits_{n=0}^\infty\frac{\left(1\right)_n\left(1-\gamma\right)_n}{\left(2-\gamma\right)_n n!}\frac{\e^{nt}}{c_1^n}}{c_1\left(\gamma-1\right)\left(c_1+\e^t\right)}\right]_{t_0}^{t_{\f}} \\
  =\left[\underbrace{\frac{-2\gamma^2\e^{\left(1-\gamma\right)t}\left(c_1\gamma+c_1+\gamma\e^t\right)}{\left(c_1+\e^t\right)^2}}_{\eqqcolon\psi_1(t)}\right]_{t_0}^{t_{\f}}-\left[\underbrace{\frac{\sum\limits_{n=0}^\infty\frac{\left(1\right)_n\left(1-\gamma\right)_n}{\left(2-\gamma\right)_n n!}2\gamma^4\e^{\left(1-\gamma\right)t}\frac{\e^{nt}}{c_1^n}}{c_1\left(\gamma-1\right)}}_{\eqqcolon\psi_2(t)}\right]_{t_0}^{t_{\f}}.
\end{equation*}
%\begin{align*}
%  = &\left(-\frac{2\gamma^2\e^{\left(1-\gamma\right)t_{\f}}}{x_1^{t_{\f}}\e^{t_{\f}}+\e^{t_{\f}}}+\frac{2\gamma^2\e^{\left(1-\gamma\right)t_0}}{x_1^{t_{\f}}\e^{t_{\f}}+\e^{t_0}}\right. \\\\
%  &\left.-\frac{\sum\limits_{n=0}^\infty\frac{\left(1\right)_n\left(1-\gamma\right)_n}{\left(2-\gamma\right)_n
%        n!}2\gamma^3\frac{\e^{\left(1-\gamma\right)t_{\f}}}{\left(x_1^{t_{\f}}\right)^n}}{x_1^{t_{\f}}\e^{t_{\f}}\left(\gamma-1\right)}+\frac{\sum\limits_{n=0}^\infty\frac{\left(1\right)_n\left(1-\gamma\right)_n}{\left(2-\gamma\right)_n
%        n!}2\gamma^3\e^{\left(1-\gamma\right)t_0}\frac{\e^{nt_0}}{\left(x_1^{t_{\f}}\right)^n\e^{nt_{\f}}}}{x_1^{t_{\f}}\e^{t_{\f}}\left(\gamma-1\right)}\right)_2
%\end{align*}
With this result the values of $\psi_i(t)$, $i=1,2$ at $t=t_0, t_\f$ are
substituted replacing the integral in equation (\ref{eq:lem_rel}) and four
summands can be regarded separately in the limit:
\begin{align*}
 \Psi_{1,t_\f}(t_0) &\coloneqq \frac{2\gamma^2\e^{\left(1-\gamma\right)t_{\f}}\left(x_1^{t_\f}\e^{t_\f}\gamma+x_1^{t_\f}\e^{t_\f}+\gamma\e^{t_\f}\right)\e^{-\gamma t_{\f}}}{\left(x_1^{t_{\f}}\e^{t_{\f}}+\e^{t_{\f}}\right)^2\left(\gamma^3\e^{-2\gamma t_0}-\gamma^3\e^{-2\gamma t_{\f}}\right)} \\
  \Psi_{1,t_0}(t_0) &\coloneqq \frac{2\gamma^2\e^{\left(1-\gamma\right)t_0}\left(x_1^{t_\f}\e^{t_\f}\gamma+x_1^{t_\f}\e^{t_\f}+\gamma\e^{t_0}\right)\e^{-\gamma t_{\f}}}{\left(x_1^{t_{\f}}\e^{t_{\f}}+\e^{t_0}\right)^2\left(\gamma^3\e^{-2\gamma t_0}-\gamma^3\e^{-2\gamma t_{\f}}\right)} \\
  \Psi_{2,t_\f}(t_0) &\coloneqq \frac{\sum\limits_{n=0}^\infty\frac{\left(1\right)_n\left(1-\gamma\right)_n}{\left(2-\gamma\right)_n n!}2\gamma^4\frac{\e^{\left(1-\gamma\right)t_{\f}}}{\left(x_1^{t_{\f}}\right)^n}\e^{-\gamma t_{\f}}}{\left(x_1^{t_{\f}}\e^{t_{\f}}\left(\gamma-1\right)\right)\left(\gamma^3\e^{-2\gamma t_0}-\gamma^3\e^{-2\gamma t_{\f}}\right)} \\
  \Psi_{2,t_0}(t_0) &\coloneqq
  \frac{\sum\limits_{n=0}^\infty\frac{\left(1\right)_n\left(1-\gamma\right)_n}{\left(2-\gamma\right)_n
   n!}2\gamma^4\e^{\left(1-\gamma\right)t_0}\frac{\e^{nt_0}}{\left(x_1^{t_{\f}}\right)^n\e^{nt_{\f}}}\e^{-\gamma
   t_{\f}}}{\left(x_1^{t_{\f}}\e^{t_{\f}}\left(\gamma-1\right)\right)\left(\gamma^3\e^{-2\gamma
   t_0}-\gamma^3\e^{-2\gamma t_{\f}}\right)}.
\end{align*}
The limit of the first two expressions $\lim_{t_0 \to
  -\infty}\Psi_{1,t_\f}(t_0)=\lim_{t_0 \to -\infty}\Psi_{1,t_0}(t_0)=0$ is
evident. The two terms $\Psi_{2,t_\f}(t_0)$ and $\Psi_{2,t_0}(t_0)$ contain a
hypergeometric series. $\Psi_{2,t_\f}(t_0)$ is absolutely convergent if
$x_1^{t_\f}>1$, $\Psi_{2,t_0}(t_0)$ is absolutely convergent if
$\left|{\e^{t_0}}/{x_1^{t_\f}\e^{t_\f}}\right|<1$ (generalized ratio
test). This is always fullfilled for $t_0$ small enough.
Therefore, $\lim_{t_0 \to -\infty}\Psi_{2,t_\f}(t_0)=\lim_{t_0 \to -\infty}\Psi_{2,t_0}(t_0)=0$.
\hfill\end{proof}

\section{Numerical Results} \label{s:num_res}
We present numerical results for
the examples investigated theoretically in the previous section. Additionally,
numerical SIM computations for a simplified realistic hydrogen combustion
mechanism are shown.

\subsection{Numerical Methods} \label{ss:num_meth}
After suitable discretization the optimization problem
(\ref{eq:op}) can be solved as a standard nonlinear programming problem
(NLP), for example via the sequential quadratic programming (SQP) method
\cite{Powell1978} or interior point (IP) methods, e.g.\
\cite{Forsgren2002}. In particular, one has to decide how to treat the
differential equation constraint and the objective functional. The easiest way
is a decoupled iterative approach, a full numerical integration
of the ODE model
with the current values of the variables subject to optimization. This
procedure is
called the sequential (or single shooting) approach since it fully decouples
simulation of the model and optimization. However, it is often beneficial to
have an ``all at once'' approach that couples simulation and optimization via
explicit discretization of the ODE constraint. This so-called simultaneous
approach has
the advantage of introducing more freedom into the optimization problem since
the differential equation model does not have to be solved exactly in each
iteration of the optimization algorithm. Especially for highly unstable ODE
problems such as (\ref{eq:op:dyn}) considered backwards in time a fully
discrete collocation approach seems
appropriate for the ODE constraint. On a predefined time grid the collocation
method constructs polynomials obeying the differential equation at a
certain number of nodes depending on its degree. For the numerical solutions
presented in this work we use a Radau-method with linear, quadratic, and cubic
polynomials, respectively, \cite{Ascher1998}.

The main difference between SQP and IP optimization methods for the solution
of an NLP is
the treatment of inequality constraints. Whereas SQP identifies the set of
active constraints in the solution, IP formally couples the constraint
violation to the objective function via a penalty term. Both methods finally
use variants of Newton's method applied to
the necessary optimality conditions, cf.\ \cite{Nocedal2006}. For the
numerical results presented in this work the NLP has been solved using the
robust interior point method implemented in IPOPT \cite{Waechter2006} including
linear algebra solvers of the HSL
routines \cite{HSL2007}. The required derivatives are computed using the open
source automatic differentiation package CppAD \cite{Bell2008}. Plots are
generated using MATLAB$^\circledR$.

\subsection{Linear Model} \label{ss:num_res_lin}
Figure \ref{f:ZWEIER} and Figure \ref{f:ZWEIERb} depict numerical solution
results of problem (\ref{eq:op}) with the linear model (\ref{eq:lin}) and
small time scale separation $\gamma=0.2$ and $\gamma=1.0$,
respectively. Solutions for the \emph{forward mode} and \emph{reverse mode}
are shown. In all cases $x_2$ is chosen as reaction progress variable
(parameterization of the SIM) and fixed at four different values:
$x_2^{t_{\f}}=2.0,$ $1.5,$ $1.0,$ $0.5$, for each of which the optimization
problem is solved to obtain the coordinate of the second variable supposed to
be located on the SIM (here slow eigenspace). The red curve is the SIM (slow
eigenspace) which is given as
the first bisectrix and the blue curves are the trajectories integrated
numerically starting from those points (blue circles) that have been computed
as solutions of the optimization problem. The red dot represents the
equilibrium point (stable fixed point). Obviously the \emph{reverse mode} gives
solutions that are significantly closer to the SIM than the \emph{forward mode}.
\begin{figure}[htb]
  \begin{center}
    \subfigure[\label{a} {\em Forward mode}: $t_0=0.0,
    t_\f=10.0$.]{\includegraphics[width=6.cm]{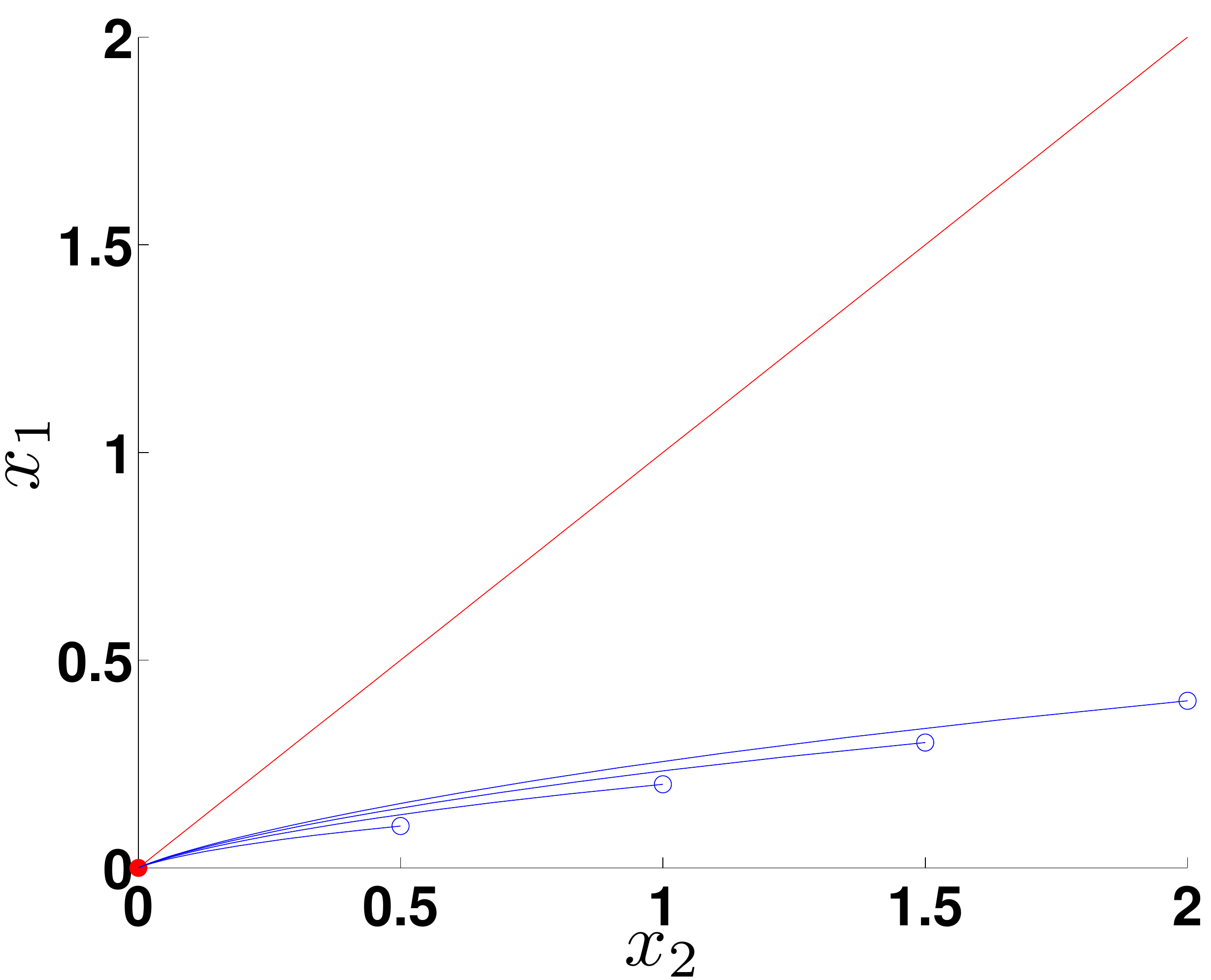}}\hfill
    \subfigure[\label{b} {\em Reverse mode}: $t_0=-21.0,
    t_\f=0.0$.]{\includegraphics[width=6.cm]{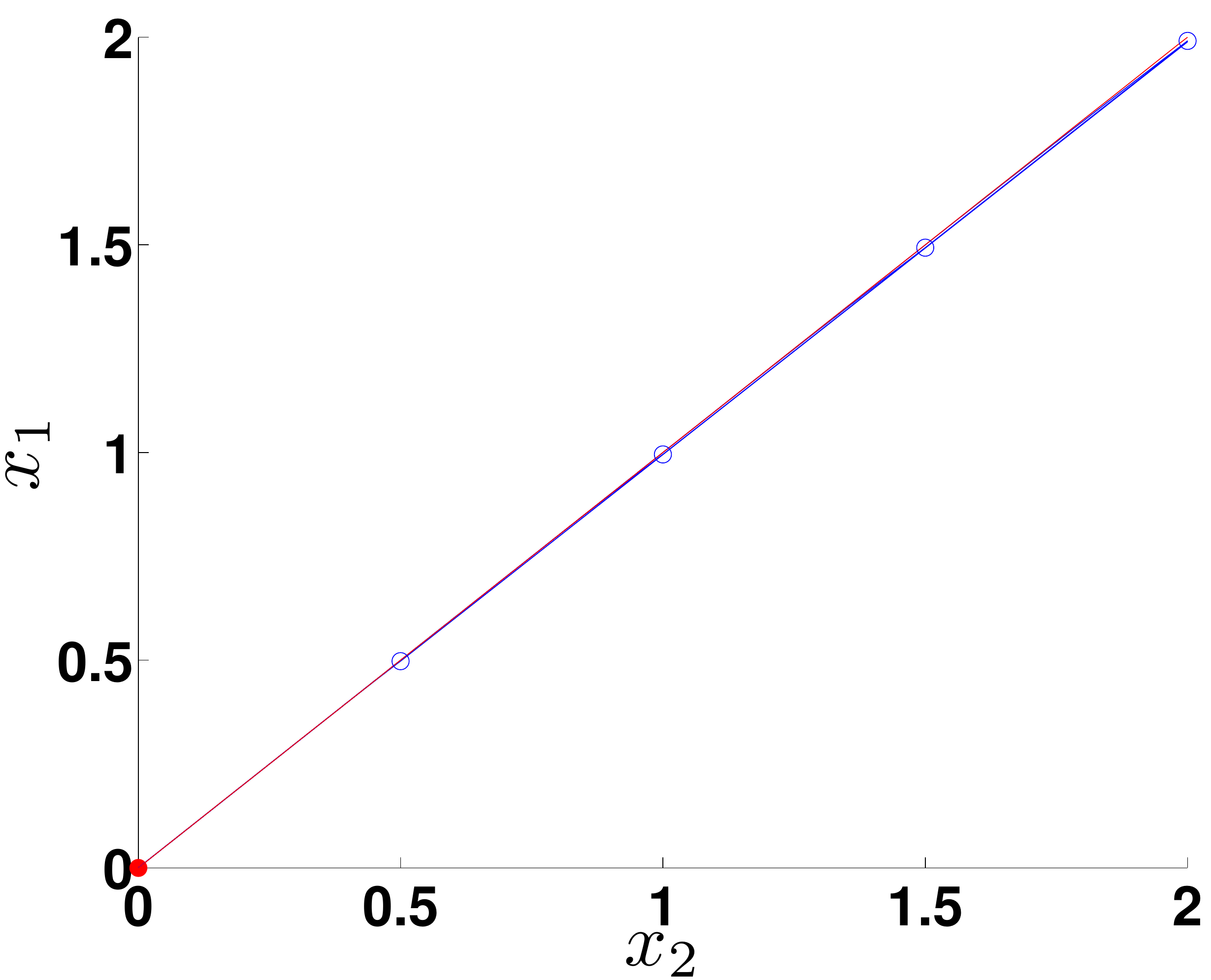}}
  \end{center}
  \caption{\label{f:ZWEIER}Results for the linear model (\ref{eq:lin_trans}) with $\gamma=0.2$,
    (a) forward mode: $x_2(t_0)=x_2^{t_0}$, and (b) backward mode:
    $x_2(t_{\f})=x_2^{t_{\f}}$.}
\end{figure}

\begin{figure}[htb]
  \begin{center}
    \subfigure[\label{aa}Forward mode with $t_0=0.0$ and
    $t_\f=10.0$.]{\includegraphics[width=6.cm]{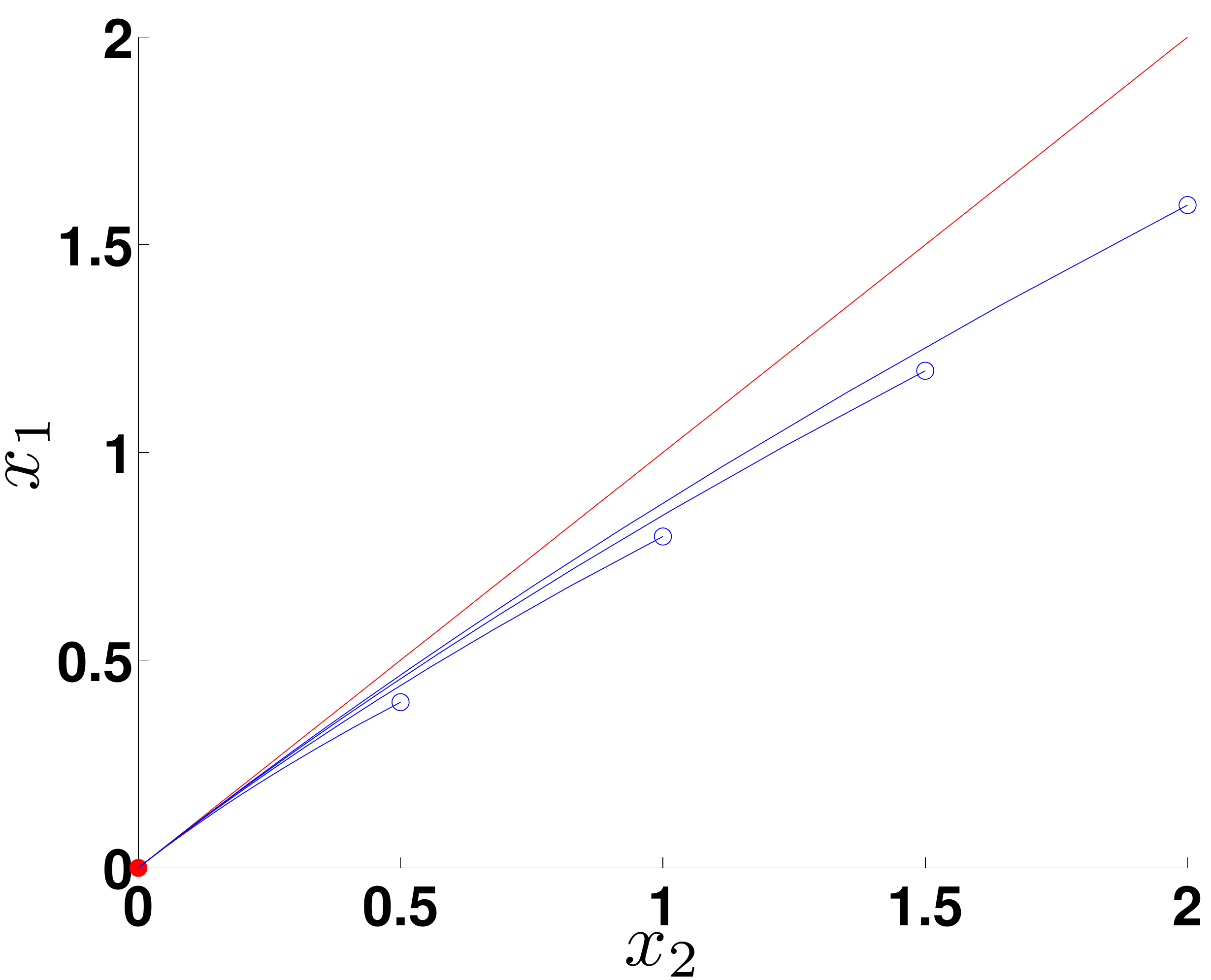}}\hfill
    \subfigure[\label{bb}Reverse mode with $t_0=-17.0$ and
    $t_\f=0.0$.]{\includegraphics[width=6.cm]{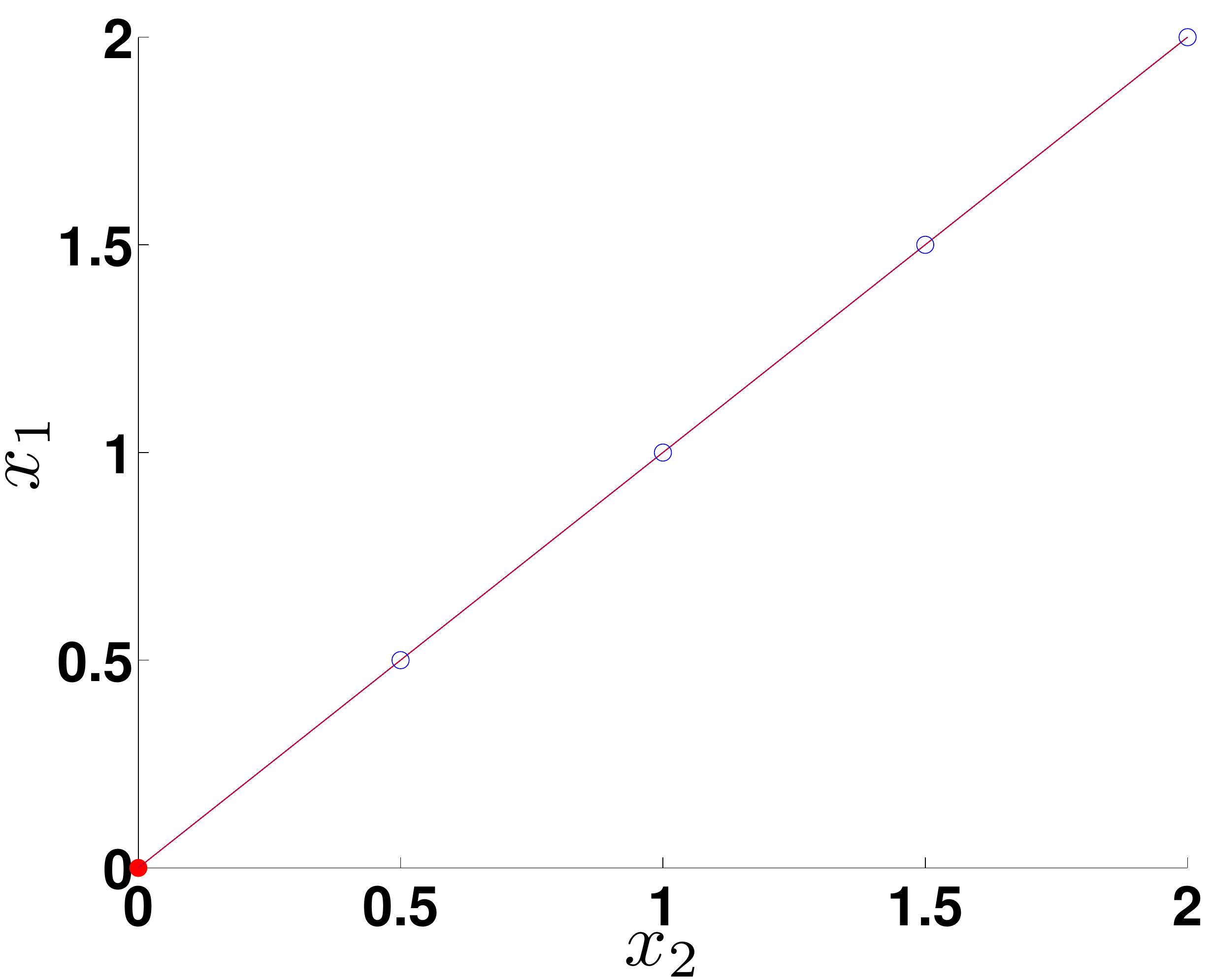}}
  \end{center}
  \caption{\label{f:ZWEIERb}Results for the linear model (\ref{eq:lin_trans}) with $\gamma=1.0$,
    (a) forward mode: $x_2(t_0)=x_2^{t_0}$, and (b) backward mode:
    $x_2(t_{\f})=x_2^{t_{\f}}$.}
\end{figure}

\subsection{Davis--Skodje Test Problem} \label{ss:num_res_DS}
\begin{figure}[htb]
  \begin{center}
    \subfigure[\label{DSa} {\em Forward mode}, $t_0=0.0,
    t_\f=10.0$.]{\includegraphics[width=6.cm]{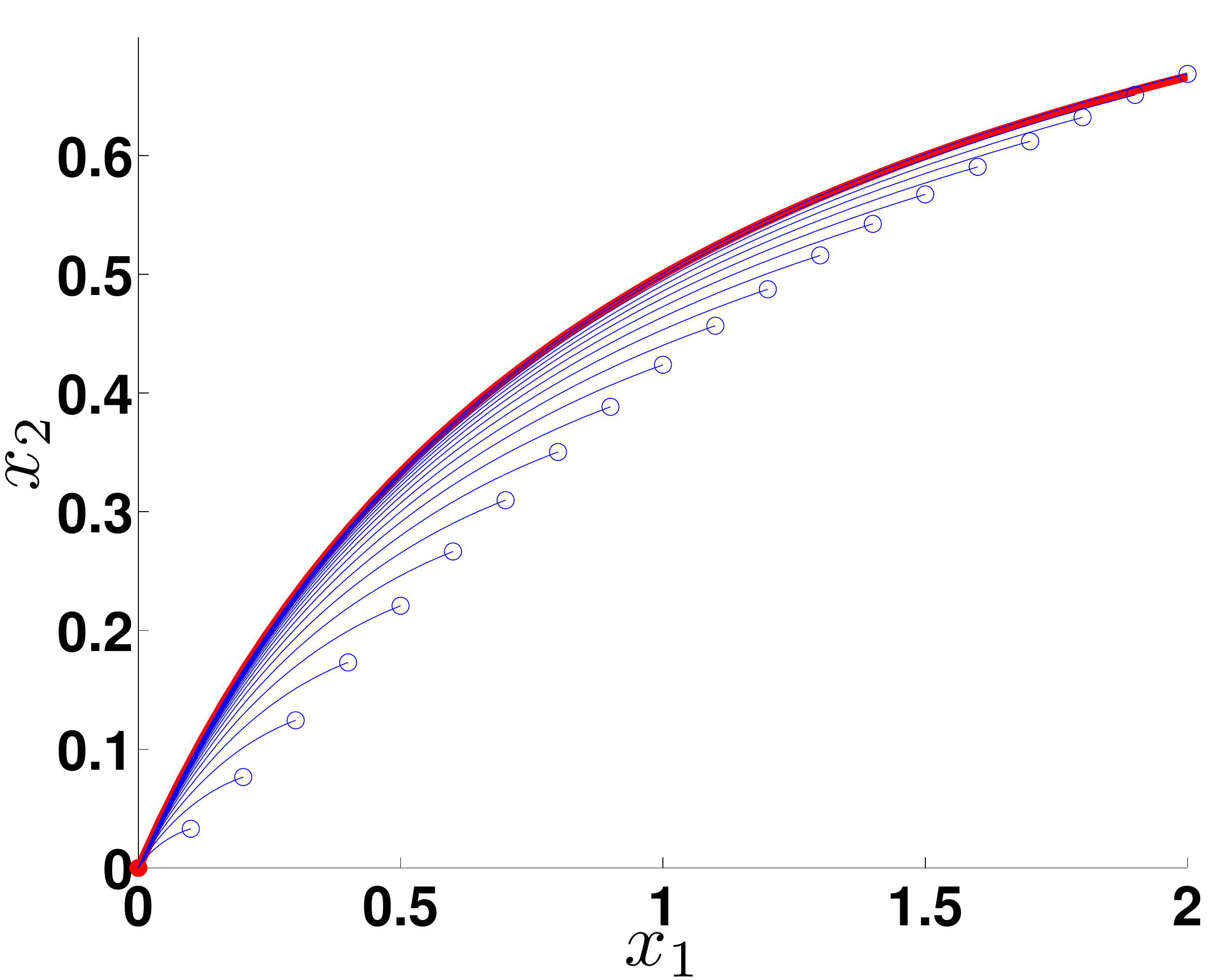}}\hfill
    \subfigure[\label{DSb} {\em Reverse mode}, $t_0=-8.0,
    t_\f=0.0$.]{\includegraphics[width=6.cm]{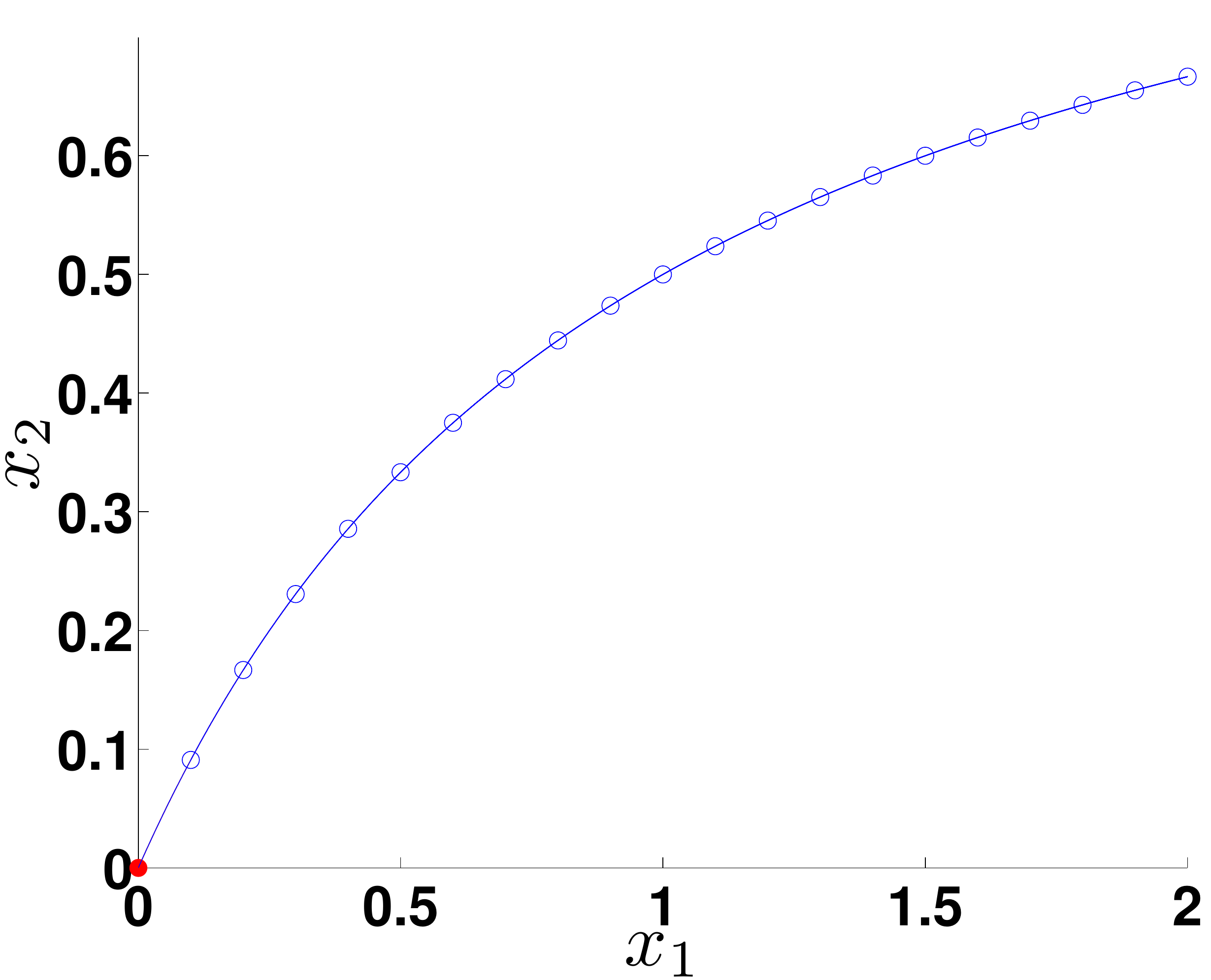}}
  \end{center}
  \caption{\label{f:DS}Results for the Davis--Skodje test problem with (a)
    forward mode: $x_1\left(t_0\right)=x_1^{t_0}$, and (b) reverse mode:
    $x_1\left(t_{\f}\right)=x_1^{t_{\f}}$. The red curve
    represents the analytically calculated SIM, the blue curves are the
    trajectories integrated numerically from those points that are the
    solutions of the optimization problem (blue circles), the red dot
    represent the chemical equilibrium point.}
\end{figure}
Similar results are shown in Figure \ref{f:DS}. Here the Davis--Skodje test
problem is used for computations with \emph{forward mode} (Fig.~\ref{DSa}) and
\emph{reverse mode} (Fig.~\ref{DSb}). In this case $x_1$ is chosen as reaction
progress variable (SIM parameterization) and fixed at several values between
$0.2$ and $2.0$ for the computation of SIM points as solutions of the
optimization problem. The spectral gap parameter is chosen as $\gamma = 1.2$.

\emph{Forward mode} solutions show a larger deviation from
the SIM (slow eigenspace) and a lack of invariance, whereas \emph{reverse mode}
solutions are highly accurate representations of the SIM.

\subsection{Simplified Realistic Mechanism} \label{ss:ice-pic}
As a third example numerical results for a simplified realistic mechanism for
hydrogen combustion are presented. The corresponding full mechanism was
originally published as a detailed hydrogen combustion mechanism by Li et al.\
in \cite{Li2004}. Ren et al.\ simplified the mechanism and used it for testing
their ICE-PIC model reduction method in
\cite{Ren2006a}. We use an adapted version of the simplified one. It consists
of six chemical species (including the inert gas $\spec{N}_2$) and twelve
chemical reactions as given in Table \ref{mech:simplhyd}. Element mass
conservation relations (in the general problem formulation equality constraints
(\ref{eq:op:cr})) for this mechanism are
\begin{align*}
  x_\mathrm{H}+2x_{\mathrm{H}_2}+x_\mathrm{OH}+2x_{\mathrm{H}_2\mathrm{O}}&=0.15 \\
  x_\mathrm{OH}+x_\mathrm{O}+x_{\mathrm{H}_2\mathrm{O}}&=0.05 \\
  2x_{\mathrm{N}_2}&=1.6.
\end{align*}
\begin{table}[htb]
  \caption{\label{mech:simplhyd}Adapted version of the simplified mechanism of \cite{Ren2006a}. Rate coefficients $k$ are computed in dependance of temperature $T$ as $k=AT^b\mathrm{exp}\left(-E_a/RT\right)$, where $R$ is the universal gas constant. In the mechanism $\spec{M}$ represents a third body being any species with collision efficiency $f_\mathrm{H}=1$, $f_{\mathrm{H}_2}=2.5$, $f_\mathrm{OH}=1$, $f_\mathrm{O}=1$, $f_{\mathrm{H}_2\mathrm{O}}=12$, and $f_{\mathrm{N}_2}=1$.}
  \centerline{
    \begin{tabular}{@{}lclrrr@{}}
      \toprule
      Reaction &&& $A$ / $\text{cm}, \text{mol}, \text{s}$ & $b$ & $E_\text{a}$ / $\frac{\text{kJ}}{\text{mol}}$ \\
      \midrule
      $\spec{O} + \spec{H_2} $ & $\rightarrow$ & $\spec{H} + \spec{OH}$ &  $5.08\times 10^{04}$ & $2.7$ & $26.3$  \\
      $\spec{H} + \spec{OH}$ & $\rightarrow$ & $\spec{O} + \spec{H_2}$ & $2.24\times 10^{04}$ & $2.7$ & $18.5$  \\
      $\spec{H_2} + \spec{OH}$ & $\rightarrow$ & $\spec{H_2O} + \spec{H}$ & $2.16\times 10^{08}$ & $1.5$ & $14.4$ \\
      $\spec{H_2O} + \spec{H}$ & $\rightarrow$ & $\spec{H_2} + \spec{OH}$ & $9.62\times 10^{08}$ & $1.5$ & $77.7$ \\
      $\spec{O}$ + $\spec{H_2O}$ & $\rightarrow$ & $\spec{2}\,\spec{OH}$ & $2.97\times 10^{06}$ & $2.0$ & $56.1$ \\
      $\spec{2}\,\spec{OH}$ & $\rightarrow$ & $\spec{O} + \spec{H_2O}$ & $2.94\times 10^{05}$ & $2.0$ & $-15.1$ \\
      $\spec{H_2} + \spec{M}$ & $\rightarrow$ & $\spec{2}\,\spec{H} + \spec{M}$ & $4.58\times 10^{19}$ & $-1.4$ & $436.7$ \\
      $\spec{2}\,\spec{H} + \spec{M}$ & $\rightarrow$ & $\spec{H_2} + \spec{M}$ & $1.18\times 10^{19}$ &$-1.4$ & $0.7$ \\
      $\spec{O} + \spec{H} + \spec{M}$ & $\rightarrow$ & $\spec{OH} + \spec{M}$ & $4.71\times 10^{18}$ & $-1.0$ & $0.0$ \\
      $\spec{OH} + \spec{M}$ & $\rightarrow$ & $\spec{O} + \spec{H} + \spec{M}$ & $8.07\times 10^{18}$ & $-1.0$ & $428.2$\\
      $\spec{H} + \spec{OH} + \spec{M}$ & $\rightarrow$ & $\spec{H_2O} + \spec{M}$ & $3.80\times 10^{22}$ & $-2.0$ & $0.0$ \\
      $\spec{H_2O} + \spec{M}$ & $\rightarrow$ & $\spec{H} + \spec{OH} + \spec{M}$ &  $6.57\times 10^{23}$ & $-2.0$ & $499.4$\\
      \bottomrule
    \end{tabular}
  }
\end{table}

In Figure \ref{f:H2C6ICErw1} results for the computation of a one-dimensional
slow invariant manifold for the hydrogen combustion mechanism are shown.
Solutions of the optimization problem (\ref{eq:op}) have been computed using the
\emph{reverse mode}. Again the red
dot represents the chemical equilibrium and the progress variable
$x_{\mathrm{H}_2\mathrm{O}}$ has been fixed at different values between $0.0005$
and $0.0180$. The blue circles are the final values $x(t_\f)$ of the solution
trajectories of the optimization problem and blue curves are the trajectories
integrated numerically from those values forward to equilibrium.
They accurately approximate the SIM; convergence of trajectories (dashed red
curves) started from arbitrary initial values (red circles) to the computed
SIM is visualized in Figure \ref{f:H2C6ICErw1}.
\begin{figure}[htb]
  \centering
  \begin{center}
    \includegraphics[width=12cm]{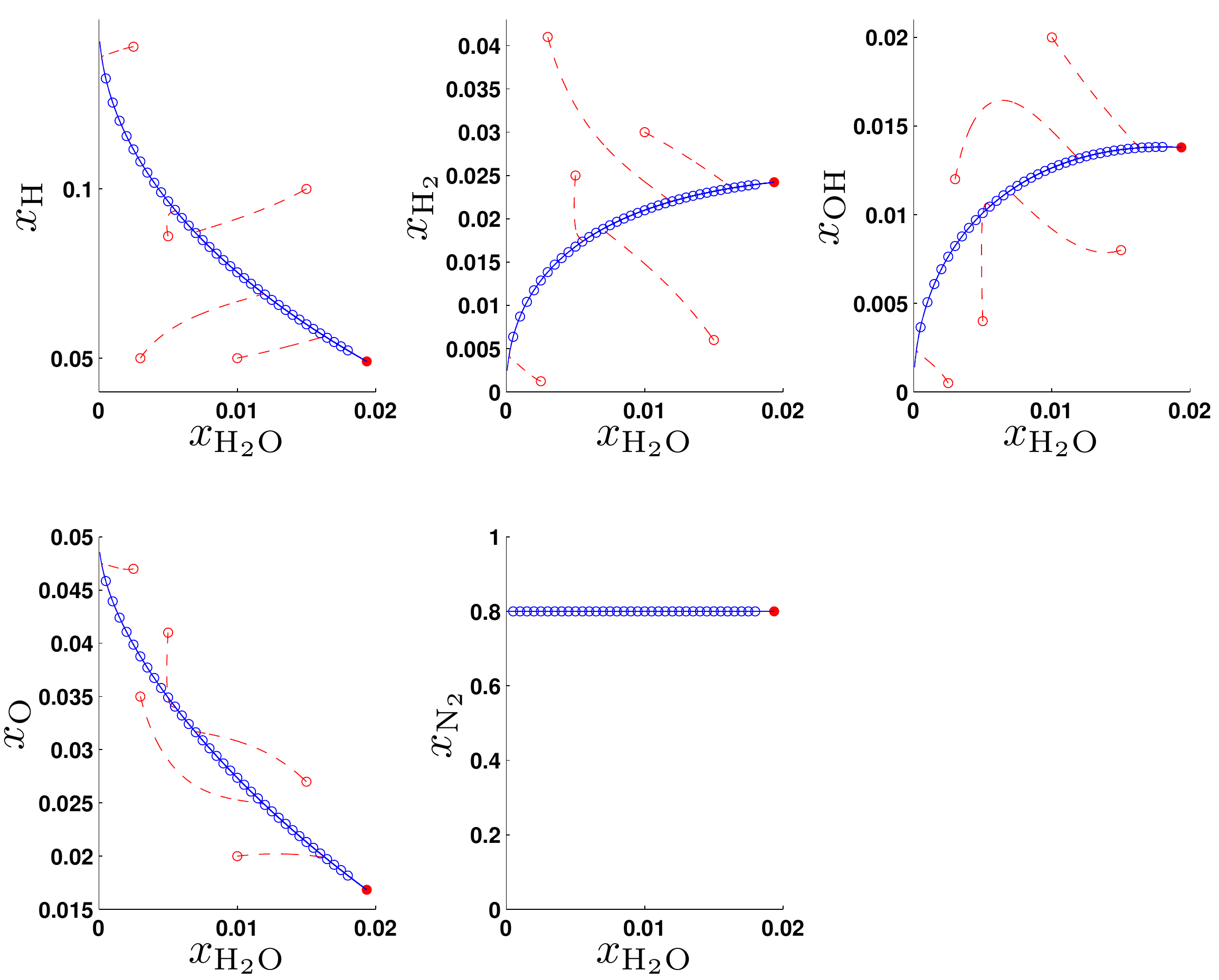}
  \end{center}
  \caption{\label{f:H2C6ICErw1}Results for a one-dimensional SIM of the simplified
    combustion mechanism computed with \emph{reverse mode},
    i.e.\
    $x_{\mathrm{H}_2\mathrm{O}}\left(t_{\f}\right)=x_{\mathrm{H}_2\mathrm{O}}^{t_{\f}}$. $t_0=-0.0004,
    t_\f= 0.0$, temperature $T=3000~{\rm K}$. Arbitrary
    trajectories relax on the manifold (red, dashed).}
\end{figure}

Figure \ref{f:ICEPIC2D} shows a two-dimensional manifold computed with the
{\em reverse mode}. Two
reaction progress variables $x_{\spec{H}_2\spec{O}}$ and $x_{\spec{H}_2}$ are fixed and the slow invariant manifold is approximated
on a two-dimensional grid as a solution of a family of optimization problems.
\begin{figure}[htb]
  \centering
  \begin{center}
    \includegraphics[width=12cm]{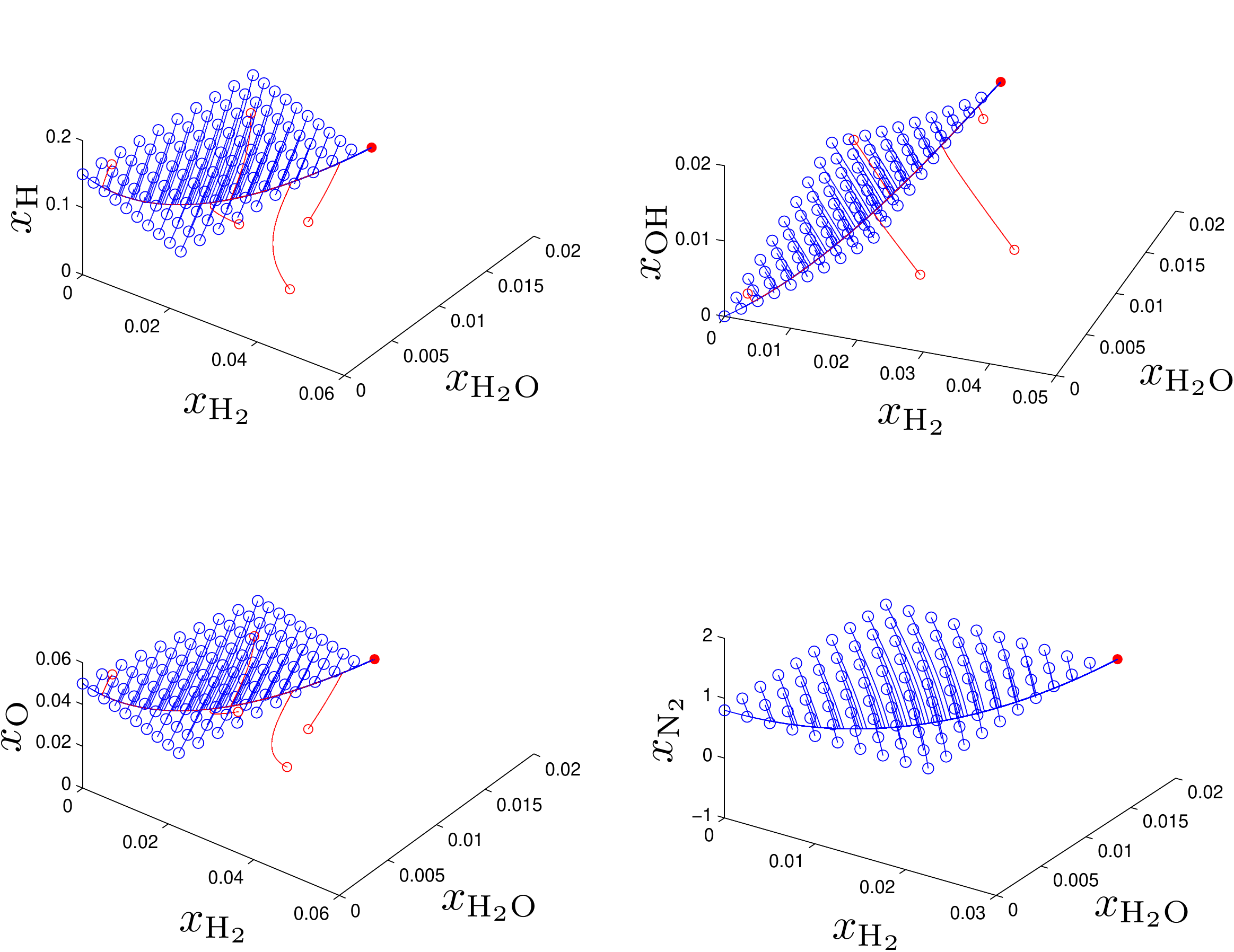}
  \end{center}
  \caption{\label{f:ICEPIC2D} Results for two-dimensional SIM of the simplified
    combustion mechanism computed with \emph{reverse mode} and
    $x_{\mathrm{H}_2\mathrm{O}}(t_{\f})$ and $x_{\mathrm{H}_2}(t_{\f})$ chosen
    as reaction progress variables,
    $t_0=-5.0\times10^{-7}, t_\f= 0.0$, constant temperature $T=3000~{\rm
      K}$.
    The same arbitrary trajectories as in Fig.~\ref{f:H2C6ICErw1} are shown in
    red.}
\end{figure}

\section*{Acknowledgments}
The authors thank Dr.~Mario~Mommer (IWR, Heidelberg) for interesting
discussions.

\clearpage
%\bibliography{./lebiedz_siads_2009}

\end{document}